\documentclass[11pt]{article}
\usepackage[utf8]{inputenc}
\usepackage[T1]{fontenc}
\usepackage[english]{babel}
\usepackage{amsmath, amssymb, amsthm}
\usepackage{mathrsfs}
\usepackage{geometry}
\usepackage{graphicx}
\usepackage{fancyhdr} 
\usepackage{tikz}
\usetikzlibrary{decorations.markings,arrows.meta}
\usepackage{amscd}
\usepackage{epsf}
\usepackage{color}
\usepackage{footmisc}
\usepackage{lipsum}
\usepackage{latexsym}
\usepackage{verbatim}
\usepackage{bbm}
\usepackage{afterpage}
\usepackage[pageref]{backref}
\usepackage[colorlinks=true, linkcolor=blue, citecolor=red, urlcolor=blue, backref=page]{hyperref}
\renewcommand*{\backrefalt}[4]{%
  \ifcase #1 %
    (Not cited.)%
  \or
    (Cited on page~\hyperlink{page.#2}{#2}.)%
  \else
    (Cited on pages~\hyperlink{page.#2}{#2}.)%
  \fi
}
\usepackage{cleveref}
\usepackage{etoolbox}
\usepackage[normalem]{ulem} 
\usepackage{thmtools} 
\usepackage{enumitem}

\newcommand{\mysubjclass}{\textup{2010} \textbf{Mathematics Subject Classification:} 53C26, 22E25, 53C29, 53C55, 17B05 }
\newcommand{\mykeywords}{\textbf{Keywords:} Hypercomplex structure, 2-step nilpotent Lie algebra, 2-step nilmanifold, Obata connection, holonomy}

\fancypagestyle{firstpagefooter}{
  \fancyhf{} 
  \fancyfoot[L]{\scriptsize
    \mysubjclass \\
    \mykeywords \\
    }
}

\setlist[enumerate,1]{label=\textbf{\arabic*.}, ref=\arabic*, leftmargin=*, itemsep=0.5ex, topsep=0.5ex}

\makeatletter
\renewenvironment{proof}[1][\proofname]{%
  \par\pushQED{\qed}%
  \normalfont\topsep6\p@\@plus6\p@\relax
  \trivlist
  \item[\hskip\labelsep
        \textbf{\textup{#1}}\@addpunct{.}]\ignorespaces
}{%
  \popQED\endtrivlist\@endpefalse
}
\makeatother

\hypersetup{
    colorlinks=true,
    linkcolor=blue,    
    citecolor=red,     
    urlcolor=blue,     
    pdftitle={Holonomy of the Obata connection on 2-step hypercomplex nilmanifolds},
    pdfauthor={Adrián Andrada, María Laura Barberis, and Beatrice Brienza}
}

\theoremstyle{plain} 
\newtheorem{theorem}{Theorem}[section]
\newtheorem{lemma}[theorem]{Lemma}
\newtheorem{proposition}[theorem]{Proposition}
\newtheorem{corollary}[theorem]{Corollary}
\newtheorem{conjecture}[theorem]{Conjecture}
\newtheorem{question}[theorem]{Question}
\theoremstyle{definition}
\newtheorem{definition}[theorem]{Definition}
\newtheorem{example}[theorem]{Example}
\newtheorem{remark}[theorem]{Remark}

\numberwithin{equation}{section}

\def\z{\mathfrak{z}}

\def\k{\mathfrak{k}}
\def\l{\mathfrak l}
\def\g{\mathfrak{g}}
\def\h{\mathfrak{h}}
\def\n{\mathfrak{n}}

\def\a{\mathfrak{a}}

\def\hcx{\{J_{\alpha}\}}

\def\R{\mathbb{R}}

\def\N{\mathbb{N}}
\def\H{\mathbb H}

\def\al{\alpha}

\def\ad{\operatorname{ad}}

\def\alt{\raise1pt\hbox{$\bigwedge$}}

\def\eqref#1{(\ref{#1})}
\title{\textbf{Holonomy of the Obata connection on 2-step hypercomplex nilmanifolds}}
\author{Adrián Andrada, María Laura Barberis, and Beatrice Brienza}
\date{}

\begin{document}

\maketitle
\begin{abstract}
We study the holonomy of the Obata connection on 2-step hypercomplex nilmanifolds. By explicitly computing the curvature tensor, we determine the conditions under which the Obata connection is flat, showing that this depends on the nilpotency step of each complex structure. In particular, we show that for 2-step hypercomplex nilmanifolds the holonomy algebra of the Obata connection is always an abelian subalgebra of \(\mathfrak{sl}(n, \mathbb{H})\) and we prove that the \(\mathbb{H}\)-solvable conjecture holds in this case. Furthermore, we provide new examples of $k$-step nilpotent hypercomplex nilmanifolds, with arbitrary $k$, which are not Obata flat.
\end{abstract}
\thispagestyle{firstpagefooter}

\section{Introduction}
Nilmanifolds, that is, compact quotients $\Gamma \backslash G$ of connected and simply connected nilpotent Lie groups $G$ by discrete subgroups $\Gamma$, have played a fundamental role in differential geometry and topology since the seminal works of Mal’cev and Mostow. They provide a natural class of spaces where algebraic and geometric properties interact tightly, offering a fair balance between algebraic manageability and rich geometric behaviour. In low dimensions, the latter allows classification results for various geometric structures (e.g., \cite{Sal,FPS, Ug}), especially when these are left-invariant. \par
Invariant complex structures have been particularly investigated, largely due to the foundational work of Salamon \cite{Sal}, where it is shown that the existence of invariant complex structures $J$ on nilmanifolds admits special features: e.g., the existence of a so-called Salamon basis \cite[Theorem 1.3]{Sal}, and the fact that the subspace \([\mathfrak{g}, \mathfrak{g}] + J[\mathfrak{g}, \mathfrak{g}]\) is always a proper ideal of the Lie algebra \(\mathfrak{g} = \operatorname{Lie}(G)\) \cite[Corollary 1.4]{Sal}. \par
In the hypercomplex setting, many of these nice behaviours fail to carry over in a unified way. Each complex structure \(J_\alpha\) (\(\alpha = 1, 2, 3\)) may satisfy suitable compatibility conditions individually, but it is not clear whether these structures collectively induce analogous global properties. For instance, it is natural to ask whether the subspace
\[
\mathfrak{g}_1^{\mathbb{H}} := [\mathfrak{g}, \mathfrak{g}] + J_1[\mathfrak{g}, \mathfrak{g}] + J_2[\mathfrak{g}, \mathfrak{g}] + J_3[\mathfrak{g}, \mathfrak{g}]
\]
is a proper ideal of $\mathfrak{g}$. This is an open question and has been explicitly conjectured in \cite{Gor1,Gor2}. The difficulty in addressing this problem reflects the more rigid nature of hypercomplex geometry compared to the complex case.  \par
A central geometric object in this setting is the Obata connection: the unique torsion-free connection that preserves all three complex structures. Its holonomy group $\operatorname{Hol}(\nabla)$ lies in $\operatorname{GL}(n, \mathbb{H})$ and plays a crucial role in understanding the geometry of hypercomplex manifolds \cite{IP, SV}. In general, explicit descriptions of $\operatorname{Hol}(\nabla)$ are rare (see, for instance, \cite{Sol,AGT,BFGV}). For nilmanifolds, however, it is known that the holonomy is contained in the commutator subgroup $\operatorname{SL}(n, \mathbb{H})$ \cite[Corollary~3.3]{BDV}. \par
In this article, we investigate left-invariant hypercomplex structures on $2$-step nilpotent Lie groups. Surprisingly, all known examples in this class are Obata-flat. \par
This naturally leads to the question: \emph{Are all 2-step nilpotent hypercomplex Lie algebras Obata-flat?} We answer this in the negative and clarify the role of the nilpotency step of each complex structure (see Sections \ref{sec:preliminaries} and \ref{sec:obatacurvature} for details). In particular, we show that if the Obata connection is not flat, then each $J_\alpha$ must be 3-step nilpotent. No such examples had been previously known. \par
In Section~\ref{sec:examples}, we construct examples of 2-step nilpotent hypercomplex Lie algebras in which one, two, or all three complex structures are 3-step nilpotent. Among these, Example~\ref{ex:nonflat} provides the first known example of a 2-step nilpotent hypercomplex Lie algebra with non-flat Obata connection. We also provide new examples of $k$-step nilpotent hypercomplex nilmanifolds, with arbitrary $k$, which are not Obata flat.\par
Turning to structural properties, we show the Gorgynian conjecture for $2$-step hypercomplex Lie algebras (Theorem~\ref{thm:Hsolvable}). \par
In Section \ref{section:holonomy}, we prove that the holonomy algebra of the Obata connection is always abelian for $2$-step nilpotent hypercomplex Lie groups, and we describe it explicitly (Theorem \ref{thm:holonomy}). This reveals a surprising rigidity: even when the Obata connection is non-flat, the holonomy remains severely restricted. We further show that the Obata holonomy group of a hypercomplex nilmanifold which is at most 2-step is trivial if and only if the nilmanifold is a torus (Theorem \ref{thm:trivialhol}). \par
In the final part of the paper, we extend our analysis to the $3$-step nilpotent case, showing that, under suitable conditions, the holonomy of the Obata connection remains abelian (Proposition~\ref{prop:holonomy2}). \\ \ \\
\noindent \textbf{Acknowledgments.} A.\ Andrada and M.\ L.\ Barberis were partially supported by CONICET, SECyT-UNC and ANPCyT (Argentina). B. Brienza was partially supported by Project PRIN 2022 “Real and complex manifolds: Geometry and Holomorphic Dynamics”, by GNSAGA (Indam) and by the INdAM - GNSAGA Project CUP E53C24001950001. The third author is grateful to A. Fino, G. Gentili, G. Gorginyan, and M. Verbitsky for many helpful discussions, and to the FAMAF for the warm hospitality during the preparation of this work. The authors would like to thank the anonymous referee for their detailed review of the manuscript.

\section{Preliminaries} \label{sec:preliminaries}
A \emph{complex structure} on a differentiable manifold $M$ is an automorphism $J$ of the tangent bundle $TM$ satisfying $J^{2} = -I$ and the integrability condition $N_J(X,Y) = 0$ for all vector fields $X,Y$ on $M$, where $N_J$ denotes the Nijenhuis tensor defined by
\[
N_J(X,Y) = [X,Y] + J([JX,Y] + [X,JY]) - [JX,JY].
\]
The integrability of $J$ is equivalent to the existence of an atlas on $M$ whose transition functions are holomorphic maps \cite{NN}. \par
A \emph{hypercomplex structure} on $M$ is given by a triple of complex structures $\{J_\alpha\}$, $\alpha=1,2,3$, satisfying the quaternionic relations
\begin{equation}\label{quat}
J_1 J_2 = -J_2 J_1 = J_3.
\end{equation}
As a consequence, $M$ admits a 2-sphere family of complex structures $J_y = y_1 J_1 + y_2 J_2 + y_3 J_3,$
parametrized by points $y = (y_1,y_2,y_3)$ in the unit sphere $S^2 \subset \mathbb{R}^3$. It follows from \eqref{quat} that at each point $p \in M$, the tangent space $T_p M$ carries a natural structure of a module over the quaternions $\mathbb{H}$. In particular, $\dim M \equiv 0 \; \pmod{4}$. \par
Given a hypercomplex structure $\{J_\alpha\}$ on a manifold $M$, there exists a unique torsion-free connection $\nabla$ on $M$ that preserves each complex structure, i.e., satisfies $\nabla J_\alpha = 0$ for $\alpha = 1,2,3$. This connection, known as the Obata connection (see \cite{Ob,Sol}), was shown in \cite{Sol} to admit the following explicit expression:
\begin{equation}\label{eqn:Obata}
\nabla_X Y = \frac{1}{2} \bigl( [X,Y] + J_1 [J_1 X, Y] - J_2 [X, J_2 Y] + J_3 [J_1 X, J_2 Y] \bigr).
\end{equation}
More generally, by applying cyclic permutations $(\alpha, \beta, \gamma)$ of $(1,2,3)$, the Obata connection can also be written as:
\begin{equation}\label{eqn:obata-alpha}
\nabla_X Y = \frac{1}{2} \bigl( [X,Y] + J_\alpha [J_\alpha X, Y] - J_\beta [X, J_\beta Y] + J_\gamma [J_\alpha X, J_\beta Y] \bigr).
\end{equation}
It is known that $(M,\hcx)$ is quaternionic (i.e., locally modelled over $\H^n$) if and only if the Obata connection $\nabla$ is flat \cite[Theorem 11.2]{Ob}. This highlights a fundamental difference with the complex case. For instance, the $K3$ surfaces and the homogeneous spaces $\operatorname{SU}(3)$ and $\operatorname{SO}(6)/\operatorname{SU}(2)$ provide examples of  hypercomplex but not quaternionic manifolds \cite{Boy,Joy}. \par
Let $G$ be a Lie group with Lie algebra $\mathfrak{g}$. A hypercomplex structure $\{J_\alpha\}$ on $G$ is said to be \emph{left-invariant} if left translations by elements of $G$ are hyperholomorphic, i.e., holomorphic with respect to each $J_\alpha$. In this case, the structure is completely determined by its value at the identity, which corresponds to a hypercomplex structure on $\mathfrak{g}$. \par
If $\Gamma$ is a lattice in $G$, any left-invariant hypercomplex structure on $G$ descends to a hypercomplex structure on the compact quotient $\Gamma \backslash G$, which is then called \emph{invariant}. In this setting, the natural projection $G \to \Gamma \backslash G$ is hyperholomorphic. \par
We recall the following definition.
\begin{definition}
A Lie algebra $\mathfrak{g}$ is said to be \emph{nilpotent} if its lower central series
\[
\mathfrak{g} = \mathfrak{g}^{0} \supset \mathfrak{g}^{1} \supset \cdots \supset \mathfrak{g}^{k} \supset \cdots,
\]
defined recursively by $\mathfrak{g}^{k} = [\mathfrak{g}, \mathfrak{g}^{k-1}]$ for $k \geq 1$, eventually becomes zero; that is, there exists some $k \in \mathbb{N}$ such that
\[
\mathfrak{g}^{k} = 0.
\]
The smallest such integer $k$ is called the \emph{nilpotency step} of $\mathfrak{g}$. 
\end{definition}

\begin{remark}
Equivalently, a Lie algebra \(\mathfrak{g}\) is nilpotent if its upper central series
\[
0 = \mathfrak{g}_{0} \subset \mathfrak{g}_{1} \subset \cdots \subset \mathfrak{g}_{k} \subset \cdots,
\]
defined recursively by
\[
\mathfrak{g}_{k} = \{ X \in \mathfrak{g} \mid [X,\mathfrak{g}] \subset \mathfrak{g}_{k-1} \}, \qquad k \geq 1,
\]
eventually coincides with \(\mathfrak{g}\).
Moreover, the smallest integer \(k\) such that \(\mathfrak{g}_{k} = \mathfrak{g}\) coincides with the nilpotency step defined via the lower central series.
\end{remark}

Accordingly, a Lie algebra $\mathfrak{g}$ is said to be \emph{2-step nilpotent} if its commutator ideal $\mathfrak{g}^{1} = [\mathfrak{g}, \mathfrak{g}]$ lies in the center $\mathfrak{z}$ of $\mathfrak{g}$. A Lie group $N$ is \emph{2-step nilpotent} if its Lie algebra $\mathfrak{g}$ is 2-step nilpotent.
\begin{definition}
A \emph{2-step nilmanifold} is a compact quotient $\Gamma \backslash N$, where $N$ is a connected, simply connected 2-step nilpotent Lie group and $\Gamma \subset N$ is a lattice.
\end{definition}
It is well known that the existence of a lattice for a connected and simply connected nilpotent Lie group $N$ is equivalent to the existence of a rational basis of its Lie algebra $\mathfrak{g}$ \cite{Mal}, that is, a basis of $\g$ for which the structure constants are rational numbers. \par
Let $(\mathfrak{g},J)$ be a nilpotent Lie algebra equipped with a complex structure $J$. In this context, it is natural to consider an ascending series adapted to $J$.
\begin{definition} [\cite{CFGU}]
The ascending series $\{\a_k\}_{k \ge 0}$ (compatible with $J$) of $\g$ is defined inductively by
\begin{equation} \label{eqn:ascending}
 \a_0 = 0,\quad  \a_k := \{X\in \g \ | \ [X,\g] \subset \a_{k-1} \ \  \text{and} \ \ [JX,\g] \subset
 \a_{k-1}\}, \ k\ge 1.  
\end{equation}
\end{definition}
The complex structure $J$ is said to be \emph{nilpotent} if there exists a $k \in \N$ such that $\a_k=\g$. The smallest such $k$ is called the \emph{nilpotency step of $J$}. \par
If $\g$ is $2$-step nilpotent, then any complex structure $J$ is nilpotent \cite[Proposition 3.3]{Rol}, and the nilpotency step of $J$ is either $2$ or $3$ \cite[Theorem 1.3]{GZZ}. We recall from \cite{B3} the following lemma.
\begin{lemma}[{\cite[Lemmas 3.3 and 4.1]{B3}}] \label{lem:nilpotencystep}
Let $\g$ be a $2$-step nilpotent Lie algebra endowed with a complex structure $J$. Then \item[]
\begin{enumerate} 
    \item [$(1)$] $J$ is $2$-step nilpotent if and only if $J \g^1 \subset \z$,
    \item [$(2)$]$J$ is $3$-step nilpotent if and only if there exists $X \in \g^1$ such that $JX \notin \z$. 
    \item [$(3)$]$J\g^1$ is an abelian ideal of $\g$. 
\end{enumerate}
\end{lemma}
\begin{proof}
We include the proof of Statement (3) for further purposes. The proof is an application of the integrability of $J$. For any $X \in \g$ and $Y \in \g^1$, $N_{J}=0$ implies that 
\begin{equation} \label{eqn:integrability2}
[J Y,X]=-J[J Y,J X].
\end{equation}
In particular, if $X \in J \g^1$, then $[J Y, X]=0$, as $\g^1 \subset \z$. 
\end{proof}

\begin{remark}
    Recall that a complex structure $J$ on a Lie algebra $\g$ is called \textit{abelian} if it satisfies $[JX,JY]=[X,Y]$ for all $X,Y\in \g$ (see \cite{BDM}). In the case when $\g$ is 2-step nilpotent, an abelian complex structure is $2$-step nilpotent, since the center is $J$-invariant.\par
    A hypercomplex structure $\hcx$ on a Lie algebra $\g$ is called abelian if $J_\alpha$ is abelian for some $\alpha=1,2,3$, and it follows from \cite[Lemma 3.1]{DF2} that $J_y$ is abelian for each $y\in S^2$.
\end{remark}
\section{Computation of the Obata curvature} \label{sec:obatacurvature}
Let $(\g,\hcx)$ be a $2$-step nilpotent hypercomplex Lie algebra. As already observed, the commutator ideal $\g^1=[\g,\g]$ lies in the center $\mathfrak{z}$ of $\g$. This special feature allows for an explicit computation of the curvature tensor of the Obata connection.
\begin{proposition} \label{prop:obatacurvature}
Let $(\g,\hcx)$ be a $2$-step nilpotent hypercomplex Lie algebra. Then, for any $X,Y,Z \in \g$ the curvature tensor of the Obata connection is given by 
\begin{align} \label{eqn:obatacurvature}
4R(X,Y)Z=&-J_\beta[X,J_\beta[Y,Z]]+J_\gamma [J_\alpha X,J_\beta[Y,Z]] \nonumber \\ \nonumber
&-J_\gamma [X,J_\beta[J_\alpha Y,Z]]-J_\beta[J_\alpha X,J_\beta [J_\alpha Y,Z]]\\ \nonumber
&-[X,J_\beta [Y,J_\beta Z]]-J_\alpha[J_\alpha X,J_\beta[Y,J_\beta Z]]           \\ \nonumber
&+J_\alpha[X,J_\beta [J_\alpha Y,J_\beta Z]]-[J_\alpha X,J_\beta [J_\alpha Y, J_\beta Z]]\\ \nonumber
&+J_\beta[Y,J_\beta[X,Z]]-J_\gamma [J_\alpha Y,J_\beta[X,Z]]\\
&+J_\gamma [Y,J_\beta[J_\alpha X,Z]]+J_\beta[J_\alpha Y,J_\beta [J_\alpha X,Z]]\\ \nonumber
&+[Y,J_\beta [X,J_\beta Z]]+J_\alpha[J_\alpha Y,J_\beta[X,J_\beta Z]]\\ \nonumber
&-J_\alpha[Y,J_\beta [J_\alpha X,J_\beta Z]]+[J_\alpha Y,J_\beta [J_\alpha X, J_\beta Z]]\\ \nonumber
&-2 [J_\beta[X,Y],J_\beta Z]+2J_\gamma [J_\beta[X,Y],J_\alpha Z],
\end{align}
where $(\alpha,\beta,\gamma)$ is any cyclic permutation of $(1,2,3)$. 
\end{proposition}
\begin{proof}
The proof is an application of the formula \eqref{eqn:obata-alpha} and the integrability of $\hcx$. Using the formula~\eqref{eqn:obata-alpha} twice, the inclusion $\mathfrak{g}^1 \subset \mathfrak{z}$ and that $\nabla_X J_\alpha Y=J_\alpha \nabla_X Y$, we compute:
\begin{align*} 
4 \nabla_X \nabla_Y Z&=-J_\beta[X,J_\beta[Y,Z]]+J_\gamma [J_\alpha X,J_\beta[Y,Z]]  \\ 
&\quad -J_\gamma [X,J_\beta[J_\alpha Y,Z]]-J_\beta[J_\alpha X,J_\beta [J_\alpha Y,Z]]\\
&\quad -[X,J_\beta [Y,J_\beta Z]]-J_\alpha[J_\alpha X,J_\beta[Y,J_\beta Z]]           \\
&\quad +J_\alpha[X,J_\beta [J_\alpha Y,J_\beta Z]]-[J_\alpha X,J_\beta [J_\alpha Y, J_\beta Z]].
\end{align*}
Similarly, applying again \eqref{eqn:obata-alpha} and using the inclusion $\mathfrak{g}^1 \subset \mathfrak{z}$, we compute:
\begin{align*} 
2 \nabla_{[X,Y]} Z&= J_\alpha [J_\alpha[X,Y],Z]+J_\gamma[J_\alpha[X,Y],J_\beta Z]\\
&= [J_\alpha[X,Y],J_\alpha Z]+[J_\beta [X,Y],J_\beta Z]\\
&\quad  -[J_\alpha[X,Y],J_\alpha Z]-J_\gamma[J_\beta[X,Y],J_\alpha Z]\\
&=[J_\beta [X,Y],J_\beta Z]-J_\gamma[J_\beta[X,Y],J_\alpha Z], 
\end{align*}
where the second equality follows from the integrability of $J_\alpha$ on the first term and the integrability of $J_\gamma$ on the second term.
\end{proof}
As a consequence of Proposition \ref{prop:obatacurvature} we obtain the following result. This is a generalization of \cite[Proposition 5.1]{B1}, where it was shown that if $\z$ is $J_\al$ invariant for all $\al$ then $\hcx$ is Obata flat. 
\begin{corollary} \label{cor:flat}
Let $(\g,\hcx)$ be a $2$-step nilpotent hypercomplex Lie algebra. If any of the three complex structures $J_1,J_2,J_3$ is $2$-step nilpotent, then $(\g,\hcx)$ is Obata flat.
\end{corollary}
\begin{proof}
If $J_\beta$ is $2$-step nilpotent then $J_\beta  \g^1 \subset \z$ (Lemma \ref{lem:nilpotencystep}). The result then follows from Proposition~\ref{prop:obatacurvature}, since  each term contains an element of the form $[\cdot,J_\beta \g^1]$.
\end{proof}

\begin{remark}
As a consequence of Corollary \ref{cor:flat}, we have that if the Obata connection associated to $(\g,\hcx)$ is not flat, then $J_\alpha$ is $3$-step for all $\al$. Moreover, any complex structure $J_y$ in the hypercomplex 2-sphere is $3$-step. 
\end{remark}
\section{Construction of non-flat 2-step hypercomplex nilmanifolds} \label{sec:examples}
To the best of our knowledge, in all known examples of $2$-step nilpotent hypercomplex Lie algebras $(\mathfrak{g}, \hcx)$, each complex structure $J_\alpha$ is $2$-step \cite{B2,DF1,DF3,BD,LW,FG}. In particular, every known example of a $2$-step nilpotent Lie group equipped with a left-invariant hypercomplex structure is Obata flat (see Corollary~\ref{cor:flat}).\par 
According to \cite[Section 3]{DF3}, in dimension $8$ any nilpotent Lie algebra admitting a hypercomplex structure is $2$-step nilpotent, and, moreover, each complex structure $J_\alpha$ is $2$-step. \par
In this section, we construct examples of $2$-step nilpotent hypercomplex Lie algebras for which one, two, or all of the complex structures are $3$-step. Moreover, we exhibit an example in which all three complex structures are $3$-step and the associated Obata connection is not flat.
\subsection*{The construction}
Let $\n$ be a $2$-step nilpotent Lie algebra and $\g:=\n\oplus \R^{r}$for some $r\in\N$. \par
Let $\mu \in \bigwedge^2 \n^* \otimes \R^{r}$ be a $2$-form on $\n$ with values in $\R^{r}$ such that $\n^1 \subset \ker(\mu)$, where $\n^1=[\n,\n]$. We use $\mu$ to define a bracket $[\cdot,\cdot]'$ on $\g$ as follows:
\begin{align*}
 &[X,Y]'=([X,Y],\mu(X,Y)), \  X,Y \in \n, \\
 &[X,Y]'=0, \  X,Y \in \R^r, \\
 &[X,Y]'=0, \  X \in \n \ \ \text{and} \ \ Y \in \R^r.
\end{align*}
Then, $(\g, [\cdot,\cdot]')$ is a $2$-step nilpotent Lie algebra. We note that the condition \(\mathfrak{n}^1 \subset \ker(\mu)\) guarantees that the bracket \([\cdot,\cdot]'\) satisfies the Jacobi identity and that \(\mathfrak{g}\) is \(2\)-step nilpotent.\par
Let us now assume that $\n$ is a $2$-step nilpotent Lie algebra \emph{endowed with a hypercomplex structure} $\hcx$, and let $(\g=\n \oplus \R^{4k}, [\cdot,\cdot]')$ be the $2$-step nilpotent Lie algebra constructed as above. Then $\g$ carries a natural almost complex structure $J_\alpha'=J_\alpha\oplus I_\alpha$, where $\{I_\alpha\}$ is the standard hypercomplex structure on $\R^{4k}$. The following result is a characterization of $N_{J'_\alpha}=0$.
\begin{lemma}
The almost complex structure $J'_\alpha$ is integrable if and only if 
\begin{equation} \label{eqn:integrability}
I_\alpha \mu(X,Y)=\mu(J_\alpha X,Y)+\mu(X,J_\alpha Y)+I_\alpha \mu (J_\alpha X, J_\alpha Y), \ X,Y \in \n.
\end{equation}
\end{lemma}

Let us now assume that the $2$-step nilpotent hypercomplex Lie algebra $(\n,\hcx)$ satisfies the condition that each $J_\alpha$ is $2$-step. If there exists $X \in \n^1$ such that $\mu (J_\alpha X, \cdot) \neq 0$, then the corresponding complex structure $J'_\alpha$ on $\g=\n \oplus \R^{4k}$ is $3$-step (Lemma \ref{lem:nilpotencystep}). Since $X \in \n^1$ we may write
\[
X= \sum_j \, [Y_j,Z_j], \ \ \ Y_j, Z_j \in \n.
\]
We define 
\[
X':= \sum_j \, [Y_j,Z_j]'=\sum_j  ([Y_j,Z_j], \mu(Y_j,Z_j)) \in \g^1.
\]
We claim that $J_\alpha' X' \notin \z(\g)$. In fact, let $W \in \n$ be such that $\mu(J_\alpha X, W) \neq 0$, then
\[
[J'_\alpha X', W]'=(0,\mu(J_\alpha X,W)) \neq 0. 
\]

We will use this procedure to construct new examples of $2$-step nilpotent hypercomplex Lie algebras where one, two or all  three complex structures are $3$-step.
\subsection*{Examples}
All $8$-dimensional hypercomplex nilpotent Lie algebras were classified in \cite{DF3}. 
We consider the $8$-dimensional $2$-step nilpotent Lie algebra \(\mathfrak{n}_8\) defined by \footnote{We adopt the following notation: $e^{i,j}=e^i \wedge e^j$.}
\begin{equation} \label{eqn:n8}
    de^i = 0, \quad i=1,\dots,7, \qquad de^8 = e^{1,2} - e^{3,4} ,
\end{equation}
which corresponds to the Lie algebra \(\mathfrak{n}_1\) in \cite[Section~3]{DF3}.
It was shown in \cite[Section~3]{DF3} (see also \cite[Section~5]{DF1}) that \(\mathfrak{n}_8\) admits a hypercomplex structure $\hcx$, defined as follows:
\begin{align} \label{eqn:Jalpha}
&J_1 e_1=e_2, \ J_1 e_3=e_4, \ J_1 e_5= e_6, \ J_1 e_7=e_8, \nonumber\\
&J_2 e_1=e_3, \ J_2 e_2=-e_4, \ J_2 e_5= e_7, \ J_2 e_6=-e_8, \\
&J_3e_1=e_4, \ J_3 e_2=e_3, \ J_3 e_5= e_8, \ J_3 e_6=e_7. \nonumber
\end{align}
In particular, since $\n_8^1=\text{span}\{e_8\}$ and the center $\z(\n_8)=\text{span}\{e_5,e_6,e_7,e_8\}$, $J_\alpha$ is $2$-step for any $\alpha=1,2,3$.\par
Consider $\g=\n_8 \oplus \R^4$ and fix the standard basis $\{e_9,e_{10},e_{11},e_{12}\}$ of $\R^4$. We point out that $\{I_\al\}$ are defined as 
\begin{align*} 
&I_1 e_9=e_{10}, \ I_1 e_{11}=e_{12}, \nonumber\\
&I_2 e_9=e_{11}, \ I_2 e_{10}=-e_{12}, \\
&I_3 e_9=e_{12}, \ I_3 e_{10}=e_{11}. \nonumber
\end{align*}
\begin{example}\label{ex:2-2-3}
We consider the following Lie bracket $[\cdot,\cdot]'$ on $\g$ given by $[\cdot,\cdot]'=[\cdot,\cdot]+ \mu$, where $\mu \in \bigwedge^2 \n_8^* \otimes \R^4$ is defined as
\[
\mu =e^{1,5} \otimes e_9+e^{2,5} \otimes e_{10}+e^{3,5} \otimes e_{11}+e^{4,5} \otimes e_{12}.
\]
Since $\mu(e_8,\cdot)=0$, and Equation \ref{eqn:integrability} is satisfied for any $\alpha=1,2,3$, the $2$-step nilpotent Lie algebra $(\g=\n_8 \oplus \R^4,[\cdot,\cdot]')$ can be endowed with the hypercomplex structure $\{J'_\al=J_\al\oplus I_\al\}$. In particular, since $\mu(J_3 e_8,\cdot)=\mu(-e_5,\cdot) \neq 0$, $J'_3$ is $3$-step nilpotent, while $J'_1,J'_2$ are $2$-step nilpotent. \par
To sum up, the $2$-step nilpotent Lie algebra $\g$ defined by the structure equations
\[
\begin{split}
&de^i=0, \ \ \   i=1,\dots,7, \ \ \ de^8=e^{1,2}-e^{3,4}, \\
&de^9=-e^{1,5},  \ \ de^{10}=-e^{2,5}, \\
&de^{11}=-e^{3,5}, \ \  de^{12}=-e^{4,5},
\end{split}
\]
can be endowed with a hypercomplex structure $\{J'_\al\}$, where $J'_1,J'_2$ are $2$-step nilpotent and $J'_3$ is $3$-step nilpotent.\par
In order to construct a $2$-step nilpotent hypercomplex Lie algebra in which two of the complex structures are $3$-step nilpotent and one is $2$-step nilpotent, we consider the same underlying vector space $\mathfrak{g} = \n_8 \oplus \mathbb{R}^4$ as above, but we modify the Lie bracket by changing the $2$-form $\mu$. Let us define on $\g$ the Lie bracket $[\cdot,\cdot]'=[\cdot,   \cdot] + \mu$, where $\mu \in \bigwedge^2 \n_8^* \otimes \R^4$ is now defined as
\[
\mu =\big(e^{1,5}+e^{1,6}\big) \otimes e_9+\big(e^{2,5}+e^{2,6}\big) \otimes e_{10}+\big(e^{3,5}+e^{3,6}\big) \otimes e_{11}+\big(e^{4,5}+e^{4,6}\big) \otimes e_{12}.
\]
Since $\mu(e_8,\cdot)=0$, and Equation \ref{eqn:integrability} is satisfied for any $\alpha=1,2,3$, the $2$-step nilpotent Lie algebra $(\g=\n_8 \oplus \R^4,[\cdot,\cdot]')$ can be endowed with the hypercomplex structure $\{J'_\al=J_\al\oplus I_\al\}$. In particular, since $\mu(J_2 e_8,\cdot)=\mu(e_6,\cdot) \neq 0$ and $\mu(J_3 e_8,\cdot)=\mu(-e_5,\cdot) \neq 0$, $J'_2$ and $J'_3$ are $3$-step nilpotent, while $J'_1$ is $2$-step nilpotent. \par
In summary, the $2$-step nilpotent Lie algebra $\g$ defined by the structure equations
\[
\begin{split}
    &de^i=0, \  \ \  i=1,\dots,7, \ \  \ \ de^8=e^{1,2}-e^{3,4},\\
    &de^9=-e^{1,5}-e^{1,6}, \  \  de^{10}=-e^{2,5}-e^{2,6},\\
    & de^{11}=-e^{3,5}-e^{3,6}, \ \    de^{12}=-e^{4,5}-e^{4,6},
\end{split}
\]
can be endowed with a hypercomplex structure $\{J'_\al\}$, where $J'_1$ is $2$-step nilpotent and $J'_2$ and $J'_3$ are $3$-step nilpotent. \par
We point out that both of examples constructed above are still Obata flat by Corollary \ref{cor:flat}. 
\end{example}
\begin{example} \label{ex:nonflat}
In this last example we construct a $2$-step nilpotent hypercomplex Lie algebra where each complex structure is $3$-step. Let us consider again $\g=\n_8 \oplus \R^4$ with bracket $[\cdot,\cdot]'$ defined as $[\cdot,\cdot]'=[\cdot,\cdot]+\mu$ where $\mu \in \bigwedge^2 \n_8^* \otimes \R^4$ has the following definition:
\[
\mu=e^{5,6}\otimes e_9-e^{6,7}\otimes e_{11}+e^{5,7}\otimes e_{12}.
\]
Since $\mu(e_8,\cdot)=0$, and Equation \ref{eqn:integrability} is satisfied for any $\alpha=1,2,3$, the $2$-step nilpotent Lie algebra $(\g=\n_8 \oplus \R^4,[\cdot,\cdot]')$ can be endowed with the hypercomplex structure $\{J'_\al=J_\al \oplus I_\al\}$. Since $\mu(J_1 e_8,\cdot)=\mu(-e_7,\cdot) \neq 0$,\  $\mu(J_2 e_8,\cdot)=\mu(e_6,\cdot) \neq 0$ and $\mu(J_3 e_8,\cdot)=\mu(-e_5,\cdot) \neq 0$, $J'_1, J'_2$ and $J'_3$ are $3$-step nilpotent complex structures. \par
To sum up, the $2$-step nilpotent Lie algebra $\g$ defined by the structure equations
\[
\begin{split}
    &de^i=0, \ \ \ i=1,\dots,7, \ \ \ de^8=e^{1,2}-e^{3,4},\\
    &de^9=-e^{5,6}, \ \   de^{10}=0, \ \ \ de^{11}=e^{6,7}, \ \    de^{12}=-e^{5,7},
\end{split}
\]
admits a hypercomplex structure $\{J'_\al\}$, where $J'_1,J'_2$ and $J'_3$ are $3$-step nilpotent. We claim that this hypercomplex Lie algebra is not Obata flat. Using Formula \eqref{eqn:Obata}:
\[
\begin{split}
R'(e_8,e_1) e_1&= \nabla'_{e_{8}}\nabla'_{e_{1}} e_1=-\frac{1}{2}\nabla'_{e_{8}}e_7=-\frac{1}{4} e_9,
\end{split}
\]
where we used that $\nabla'_{e_{8}}e_1=0$ and $[e_1,e_8]=0$. \par
In Section \ref{section:holonomy}, we will prove that if we consider the corresponding $2$-step nilpotent hypercomplex Lie group, then the holonomy algebra $\mathfrak{hol}(\nabla')$ is abelian. 
\end{example}
We point out that the $2$-step nilpotent Lie groups corresponding to the three Lie algebras constructed in this section admit lattices, since all the bases considered in the examples are rational.
\subsection*{Semidirect product construction}
Next, we present a construction which is a slight generalization of \cite[Lemma~3.1]{BF}. \par
Let $(\g , \hcx )$ be an arbitrary hypercomplex Lie algebra and let $\{I_\al \}$ be the standard hypercomplex structure on $\R^{4k}$. Let $\rho:\g \to \g\l (k,\H )$ be a Lie algebra homomorphism, that is, 
\[  \rho ([X,Y])=[\rho (X), \rho(Y)], \quad \rho(X) I_\al={I_\al}_{\, }\rho (X), \;\; \text{ for all } X, \, Y \in \g, \; \al =1,2,3.   \]
Consider the Lie algebra $\h:= \g\ltimes_\rho \R^{4k}$, that is, 
\[{[(X,V),(Y,W)]}= ([X,Y],\rho(X)W-\rho(Y)V ), \quad X, Y\in \g, \; V,W\in \R ^{4k}.\]
The Jacobi identity on $\h$ follows from the fact that $\rho$ is a Lie algebra homomorphism. 
Note that $\{(0,V) \mid V \in \R^{4k}\}$ is an abelian ideal of $\h$. Let $\widetilde{J}_\al$, $\alpha =1,2,3$, be the endomorphisms of $\h$ defined by $\widetilde{J}_\al =J_\al \oplus I_\al$. 
\begin{proposition}\label{prop:semid-prod}
    Let $(\g , \hcx )$ and $(\h, \{\widetilde{J}_\al \})$ be as above, then $\{\widetilde{J}_\al \}$ is a hypercomplex structure on $\h$. Moreover, the Obata connection $\widetilde{\nabla}$ of $ \{\widetilde{J}_\al\} $ and the curvature tensor $\widetilde{R}$ of $\widetilde{\nabla}$ are given by:
\begin{eqnarray*}
\widetilde{\nabla}_{(X,V)} (Y,W)& =&(\nabla _X Y , \rho(X) W), \\ \widetilde{R}\left((X,V), (Y,W)\right) (Z,U) &=&
\left( R(X,Y)Z, 0\right), \quad X,Y,Z\in \g, \; V,W,U\in \R^{4k}, 
\end{eqnarray*}
where $\nabla$ is the Obata connection of $\hcx$ and $R$ is the curvature tensor of $\nabla $.  In particular, $\mathfrak{hol}(\widetilde{\nabla}) \cong \mathfrak{hol}(\nabla)$. 
\end{proposition}
\begin{proof}
 Since $\hcx$ is integrable on $\g$ and $\rho(X)$ commutes with $I_\al$ for all $X\in \g$, $\al =1,2,3$, it turns out that $\{\widetilde{J}_\al \}$ is integrable on $\h$. \par
The rest of the proof follows the lines of \cite[Lemma~3.1]{BF}.
\end{proof}

\begin{lemma}\label{lem:center-comm} \item[ ]
\begin{enumerate} 
    \item[$(1)$] The center and commutator ideal of the Lie algebra $\h$ are given by:
    \begin{eqnarray*}
 \z(\h)&=&\left\{  (X,V) \mid X\in \z(\g)\cap \ker \rho , \; V\in \bigcap_{\substack{Y\in \g}} \ker \rho (Y) \right\} , \\   
 \h ^1&=& \left\{  (X,V) \mid X\in \g ^1, \; V\in \rho (\g)\R^{4k}\right\} ,
    \end{eqnarray*}
where $ \rho (\g)\R^{4k}=\mathrm{span}\{ \rho(X) V \mid X \in \g \ \text{and} \  V \in \R^{4k}\}$.
\item[$(2)$] If $\g$ is $l$-step nilpotent and the set
\begin{equation}\label{eqn:minimum} \{j\in \N \mid \rho(X_1)\cdots \rho(X_j)=0, \ \text{for all} \ \ X_1,\dots,X_j \in \g\}\end{equation}
is non-empty, then $\h$ is $r$-step nilpotent, with $r=\max\{l, m_\rho\}$, where $m_\rho$ is the minimum of the set defined in \eqref{eqn:minimum}. In particular, if $\g$ is $2$-step nilpotent and $\rho(X_1)\cdot \rho(X_2)=0$ for any $X_1,X_2 \in \g$, then $\h$ is $2$-step nilpotent.
\item[$(3)$]  If $\h$ is 2-step nilpotent, then $\widetilde{J}_\al$ is 2-step nilpotent if and only if $J_\al \g^1\subset \z(\g)\cap \ker \rho$.  In particular, if $J_\al$ is 3-step nilpotent, then $\widetilde{J}_\al$ is 3-step nilpotent.
\end{enumerate}
\end{lemma}
\begin{proof}
The only nontrivial points are the second equality in~(1) and statement~(2).

The first claim follows from the computation
\[
\mathfrak{h}^1=[\mathfrak{h},\mathfrak{h}]
=[(\mathfrak{g}\oplus \mathbb{R}^{4k},\,\mathfrak{g}\oplus \mathbb{R}^{4k})]
=\mathfrak{g}^1 \oplus \rho(\mathfrak{g})\,\mathbb{R}^{4k}.
\]
To prove the second statement, we recall that $\mathfrak{h}$ is $r$-step
nilpotent if and only if $r$ is the smallest integer such that, for any
$X_i\in\mathfrak{g}$ and $V_i\in\mathbb{R}^{4k}$,
\[
\prod_{i=1}^r \operatorname{ad}_{(X_i,V_i)}=0.
\]
We first observe that, for any $s\ge 1$, 
\[
\left(\prod_{i=1}^s \operatorname{ad}_{(X_i,0)}\right)(Y,W)
=
\left( \left(\prod_{i=1}^s \operatorname{ad}_{X_i}\right) Y,\;
       \left(\prod_{i=1}^s \rho(X_i) \right)W \right).
\]
Since $\mathfrak{g}$ is $l$-step nilpotent and $m_\rho$ is minimal with
the property above, this immediately implies that
\[
r \ge \mathrm{max}\{l, m_\rho\}.
\]
On the other hand, let $r=\max\{l, m_\rho\}$. 
Let $Y \in \g$ and $W \in \R^{4k}$. Then, using the fact that
\[
\left( \prod_{i=1}^r \operatorname{ad}_{(X_i,V_i)}\right)(Y,W) 
\in \g^r \oplus \rho(\mathfrak{g})^{\,r}(\mathbb{R}^{4k}),
\]
and the definition of $l$ and $m_\rho$, 
we conclude that
\[
\prod_{i=1}^r \operatorname{ad}_{(X_i,V_i)}=0.
\]
Therefore, $\mathfrak{h}$ is $r$-step nilpotent with
$r=\max\{l, m_\rho\}$. 
\end{proof}
\noindent New examples of Obata non flat 2-step nilpotent hypercomplex Lie algebras can be obtained by applying the semidirect product construction, as we show next.  
   \begin{example}\label{ex:semidirect}
       Let $\g$ be the Lie algebra from Example \ref{ex:nonflat} with the hypercomplex structure considered there. We define  $\rho:\g\to \g\l (2,\H)\subset \mathfrak{gl}(8,\R)$ by:
       \[ \rho(e_1)=\begin{bmatrix} 
       0_4 &0_4\\
        I_4 & 0_4\end{bmatrix}, \qquad\quad \rho(e_i)=0, \quad i\geq 2 ,\]
        where $I_4$ and $0_4$ are the real $4\times 4$ identity and zero matrices, respectively. 
        It follows from Example \ref{ex:nonflat}, Proposition \ref{prop:semid-prod} and Lemma \ref{lem:center-comm} that the hypercomplex structure $\{\widetilde{J}_\al \}$ on $\h$ is not Obata flat.  Theorem \ref{thm:holonomy} in Section \ref{section:holonomy} below implies that the holonomy algebra $\mathfrak{hol}(\widetilde{\nabla})$ of the corresponding $2$-step nilpotent hypercomplex Lie group is abelian. 
   \end{example}
   \begin{example}\label{ex:k-step}
   Let $\g$ be as in Example \ref{ex:nonflat} and set $\rho:\g\to \g\l (k,\H)\subset \mathfrak{gl}(4k,\R)$, $k\geq 3$, defined by: 
        \[ \rho(e_1)=\left[
\begin{array}{ccccc}
   0_4  &  0_4 & \cdots & \cdots & 0_4 \\
   I_4 & 0_4 & 0_4 & \cdots & 0_4 \\
   0_4 & I_4 & 0_4 & \cdots & 0_4 \\
   \vdots & \vdots & \ddots & \ddots & 0_4 \\
   0_4 & 0_4 & 0_4 & I_4 & 0_4
\end{array}\right] \in M_{4k}(\R), \qquad \rho(e_i)=0, \quad i\geq 2 .\]
In this case, $\h$ is $k$-step nilpotent and $\{\widetilde{J}_\al\}$ is not Obata flat. By Proposition \ref{prop:semid-prod}, the holonomy algebra coincides with the holonomy algebra of the Example \ref{ex:nonflat}. In particular, it is abelian (see Section \ref{section:holonomy}).
\end{example}

We note that the simply connected Lie groups corresponding to the Lie algebras from examples \ref{ex:semidirect} and \ref{ex:k-step} admit lattices, since all the bases considered in the examples are rational.
\section{$\mathbb{H}$-solvability of $2$-step nilpotent Lie algebras}
Let $\g$ be a Lie algebra endowed with a hypercomplex structure $\{J_1,J_2,J_3\}$. 
Given a vector subspace $W$ of $\g$, we denote by $\H W$ the $\hcx$-invariant subspace $W +J_1 W +J_2 W +J_3 W$. 

Inspired by \cite[Definition 1.5]{Gor1}, we introduce the following definition.
\begin{definition} 
Let $\g$ be a Lie algebra endowed with a hypercomplex structure $\{J_1,J_2,J_3\}$.
Define inductively $\hcx$-invariant subspaces:
\[ \g_0^{\H}=\g, \qquad \g_k^{\H}:=\H \g_{k-1}^{\H}, \quad k\geq 1.\] 
A hypercomplex  Lie algebra $(\g , \hcx )$ is called \emph{$\H$-solvable} if the following sequence eventually becomes zero for some $k\in\N$:
\begin{equation}\label{eqn:Hsol}
\g_1^{\H}\supset\g_2^{\H}\supset\cdots\supset\g_{k-1}^{\H}\supset\g_k^{\H}=0.
\end{equation}
The smallest such $k$ is called the $\H$-solvability step of $(\g,\hcx)$.
\end{definition}
\begin{lemma}
Let $(\g,\hcx)$ be an arbitrary hypercomplex Lie algebra. 
\begin{enumerate}
    \item[$(1)$] For any $k \in \N$, $\g_k^\H$ is a subalgebra of $\g$ and it is an ideal in $\g_{k-1}^\H$.
    \item[$(2)$] If $(\g, \hcx)$ is $\H$-solvable then $\g$ is solvable. 
\end{enumerate}
\end{lemma}
\begin{proof}
We prove the first item. It suffices to show that $\mathfrak{g}_k^{\mathbb{H}}$ is an ideal of
$\mathfrak{g}_{k-1}^{\mathbb{H}}$.
Since $\mathfrak{g}^1 \subset \mathfrak{g}_1^{\mathbb{H}}$, we have that 
$\mathfrak{g}_1^{\mathbb{H}}$ is clearly an ideal of
$\mathfrak{g}_0^{\mathbb{H}}=\mathfrak{g}$.
Assume by induction that $\mathfrak{g}_{k-1}^{\mathbb{H}}$ is an ideal of
$\mathfrak{g}_{k-2}^{\mathbb{H}}$. In particular,
$\mathfrak{g}_{k-1}^{\mathbb{H}}$ is a subalgebra of $\mathfrak{g}$.
Then
\begin{align*}
    [\g_{k}^\H, \g_{k-1}^\H]=[\H[\g_{k-1}^\H,\g_{k-1}^\H], \g_{k-1}^\H]\subset [\H\g_{k-1}^\H, \g_{k-1}^\H]=[\g_{k-1}^\H, \g_{k-1}^\H] \subset \g_k^\H.
\end{align*}

To show the second, we work by induction by showing that $\g_k^\H \supset D^k $, where $\{D^k\}_{k \ge 1}$ is the derived series of $\g$.\par 
For $k=1$, $\g_1^\H\supset \g^1=D^1$. Assume that $\g_{k-1}^\H\supset D^{k-1}$. Then 
\[
\g_k^\H= \H [\g_{k-1}^\H, \g_{k-1}^\H]\supset \H [D^{k-1}, D^{k-1}]=\H D^k \supset D^k.
\]
Therefore, the $\H$-solvability of $(\g, \hcx)$ implies the solvability of $\g$.
\end{proof}
Let $(\mathfrak{g}, J)$ be a nilpotent Lie algebra endowed with a complex structure $J$. It was shown by Salamon~\cite[Corollary 1.4]{Sal} that the ideal $\mathfrak{g}^1 + J\mathfrak{g}^1$ is a proper ideal of $\mathfrak{g}$. This naturally raises the question of whether an analogous property holds for nilpotent Lie algebras endowed with a hypercomplex structure. \par
The condition that $\g_1^\H=\mathfrak{g}^1 + J_1 \mathfrak{g}^1 + J_2 \mathfrak{g}^1 + J_3 \mathfrak{g}^1$ is a proper ideal of $\mathfrak{g}$ is necessary to the $\mathbb{H}$-solvability of $(\mathfrak{g}, \hcx)$. This observation motivates the following conjecture:
\begin{conjecture}[\cite{Gor1}] \label{conj:Hsolvable}
Let $(\mathfrak{g}, \hcx)$ be a nilpotent hypercomplex Lie algebra. Then it is $\mathbb{H}$-solvable.
\end{conjecture}
When the Lie algebra $\g = \operatorname{Lie}(G)$ is $\mathbb{H}$-solvable, it turns out that, for a generic complex structure $J_y$ in the hypercomplex $2$-sphere, the complex nilmanifold $(\Gamma \backslash G, J_y)$ does not admit any complex curves \cite{Gor1}. \par
The conjecture holds for Lie algebras of hypercomplex nilmanifolds with flat Obata connection \cite{Gor2}, as well as for Lie algebras of abelian hypercomplex nilmanifolds \cite{AV}. The filtration \eqref{eqn:Hsol} corresponds to an iterated hypercomplex toric bundle \cite{Gor1}, and in the latter case it was shown in \cite[Proposition~4.5]{AV} that such iterated torus bundle terminates at a point. Moreover, for a general complex structure $J_y$, the associated complex nilmanifold $(\Gamma \backslash G, J_y)$ not only admits no complex curves (as a consequence of \cite{Gor1}), but in fact all of its subvarieties are trianalytic.

Here we prove the conjecture for $2$-step nilpotent Lie algebras.
\begin{theorem} \label{thm:Hsolvable}
Let $(\g,\hcx)$ be a $2$-step nilpotent hypercomplex Lie algebra. Then $(\g,\hcx)$ is $\H$-solvable. In particular, $\g_1^\H$ is a proper ideal of $\g$ and the $\H$-solvability step is at most $3$.
\end{theorem}
\begin{proof}
Let us consider 
\[
\g_2^\H=\H[\g_1^\H,\g_1^\H].
\]
Using that $\g^1\subset \z$, and Lemma \ref{lem:nilpotencystep}
 (3), we obtain that $[J_\alpha \g^1,J_\alpha \g^1] = 0$, and
\begin{align*}
[\g_1^\H,\g_1^\H]=&[\g^1+J_1 \g^1+J_2 \g^1+J_3 \g^1,\g^1+J_1 \g^1+J_2 \g^1+J_3 \g^1]\\
=&[J_1 \g^1,J_2 \g^1]+[J_1 \g^1,J_3 \g^1]+[J_2 \g^1,J_3 \g^1]. 
\end{align*}
Therefore,
\begin{align*}
\H [\g_1^\H,\g_1^\H]=& \   [J_1 \g^1,J_2 \g^1]+[J_1 \g^1,J_3 \g^1]+[J_2 \g^1,J_3 \g^1]\\
&+J_1([J_1 \g^1,J_2 \g^1]+[J_1 \g^1,J_3 \g^1]+[J_2 \g^1,J_3 \g^1])\\
&+J_2([J_1 \g^1,J_2 \g^1]+[J_1 \g^1,J_3 \g^1]+[J_2 \g^1,J_3 \g^1])\\
&+J_3([J_1 \g^1,J_2 \g^1]+[J_1 \g^1,J_3 \g^1]+[J_2 \g^1,J_3 \g^1]).
\end{align*}
Since by Equation \eqref{eqn:integrability2} we have that $J_\alpha[J_\alpha \g^1, J_\beta \g^1]\subset \g^1$, we may write:
\begin{align*}
\g^\H_2=\H [\g_1^\H,\g_1^\H]& \subset \g^1 +J_1 [J_2 \g^1,J_3 \g^1]+J_2[J_1 \g^1,J_3 \g^1]+J_3[J_1 \g^1,J_2 \g^1].
\end{align*}
We claim that $J_1 [J_2 \g^1,J_3 \g^1]=J_2[J_1 \g^1,J_3 \g^1]=J_3[J_1 \g^1,J_2 \g^1]$.
Using the quaternionic identities and Equation \eqref{eqn:integrability2} with $J_1$, we have:
\[
J_2[J_1 \g^1,J_3 \g^1]=J_3J_1[J_1 \g^1,J_3 \g^1]=J_3[J_1 \g^1,J_2 \g^1],
\]
and, similarly, using the integrability of $J_2$,  
\[
J_3[J_1 \g^1,J_2 \g^1]=J_1J_2[J_1 \g^1,J_2 \g^1]=J_1[J_3 \g^1,J_2 \g^1].
\]
Thus, all three summands coincide and we obtain
that $\g^\H_2 \subset \g^1+J_1 [J_2 \g^1,J_3 \g^1]\subset \g^1+J_1\g^1$. Therefore, 
\[
\g_3^\H=\H [\g^\H_2,\g^\H_2]\subset \H [J_1 \g^1,J_1 \g^1]=0,
\]
since $J_1 \mathfrak{g}^1$ is abelian by Lemma~\ref{lem:nilpotencystep}.
\end{proof}
\section{Holonomy of 2-step hypercomplex nilmanifolds} \label{section:holonomy}
We prove a structural result about the holonomy algebra of the Obata connection on $2$-step nilpotent Lie groups endowed with invariant hypercomplex structures.
\begin{theorem} \label{thm:holonomy}
Let $G^{4n}$ be a $2$-step nilpotent Lie group endowed with an invariant hypercomplex structure $\hcx$. Then the holonomy Lie algebra $\mathfrak{hol}(\nabla)$ is abelian, where $\nabla$ is the Obata connection of $(G, \hcx)$. \par
Moreover, $A \cdot B=0$ for any $A,B \in \mathfrak{hol}(\nabla)$, therefore, $\mathfrak{hol}(\nabla)$  is  a commutative associative algebra with trivial product.
\end{theorem}
\begin{proof}
The proof is divided into $4$ steps. We start by proving 
\begin{itemize}
\item \textbf{Step 1}: for any $X,Y,Z \in \g$, $R(X,Y)Z \in  \g^1 \cap J_2 \g^1+J_1 \g^1 \cap J_3 \g^1$. \par
For $(\alpha,\beta,\gamma)=(1,2,3)$, each term appearing in  Formula~\eqref{eqn:obatacurvature}  falls into one of four distinct types:
\begin{enumerate}
    \item $J_2 [\cdot, J_2[\cdot,\cdot]]=[J_2 \cdot,J_2[\cdot,\cdot]] \in \g^1 \cap J_2 \g^1$, where the equality follows from Equation \eqref{eqn:integrability2},
    \item $[\cdot, J_2[\cdot,\cdot]] \in \g^1 \cap J_2 \g^1$, as $J_2 \g^1$ is an ideal by Lemma \ref{lem:nilpotencystep},
    \item $J_3 [\cdot, J_2[\cdot,\cdot]]=J_1J_2 [\cdot, J_2[\cdot,\cdot]]=J_1[J_2 \cdot,J_2[\cdot,\cdot]] \in J_1\g^1 \cap J_3 \g^1$, where the first equality follows by the quaternionic identities and the second equality follows from Equation \eqref{eqn:integrability2}, 
    \item $J_1 [\cdot, J_2[\cdot,\cdot]]=-J_3J_2 [\cdot, J_2[\cdot,\cdot]]=-J_3[J_2 \cdot,J_2[\cdot,\cdot]] \in J_1\g^1 \cap J_3 \g^1$, where the first equality follows by the quaternionic identities and the second equality follows again from Equation \eqref{eqn:integrability2}. 
\end{enumerate}
We set $\h:=\g^1 \cap J_2 \g^1+J_1 \g^1 \cap J_3 \g^1$. It is not difficult to note that $\h$ is $\hcx$-invariant and $\h \subset \g_1^\H \subsetneq \g$.
\item \textbf{Step 2}: for any $X \in \h$, $\nabla X=0$. \par
We use the following:
\begin{lemma} \label{lemma:preliminarlemma}
For any $X \in \h$ and $Y \in \g$, 
\begin{equation} \label{eqn:obataparallel}
[Y,X]+J_1[J_1 Y,X]=0
\end{equation}
\end{lemma}
\begin{proof}
Let $X \in \h$. Then $X=X_1+X_2$, where $X_1 \in \g^1 \cap J_2 \g^1$ and $X_2 \in J_1 \g^1 \cap J_3 \g^1$. Since $X_1 \in \g^1$,  
\begin{align*}
[Y,X]+J_1[J_1 Y,X] & =[Y,X_2]+J_1[J_1 Y,X_2] \\
& = [Y,X_2]-[Y,X_2]+[J_1Y,J_1X_2]-J_1[Y,J_1X_2] \\
& = 0
\end{align*}
since $J_1X_2\in \g^1\subset \z$, where we have used the integrability of $J_1$ in the second equality.
\end{proof}
Let $Y \in \g$ and $X \in \h$. Applying Formula \eqref{eqn:Obata}
\begin{align*}
2\nabla_{Y} X &=[Y,X]+J_1[J_1 Y,X]-J_2[Y,J_2X]+J_3[J_1Y,J_2X]\\
&= [Y,X]+J_1[J_1 Y,X]-J_2([Y,J_2X]+J_1[J_1Y,J_2X])\\
&= 0,
\end{align*}
where we used Lemma \ref{lemma:preliminarlemma} twice and the $\hcx$-invariance of $\h$.
\item \textbf{Step 3:} Expression for the curvature endomorphisms.
Using Steps~1 and~2, we conclude that for any $X,Y\in\mathfrak{g}$,
\[
R(X,Y) \in \bigl\{ A \in \mathfrak{gl}(n,\mathbb{H}) \mid
\operatorname{Im}(A) \subset \mathfrak{h} \subset \ker(A) \bigr\}.
\]
Moreover, since $R(X,Y)^2=0$, we have that $R(X,Y)\in \mathfrak{sl}(n,\mathbb{H})$ and therefore 
\[
R(X,Y) \in \bigl\{ A \in \mathfrak{sl}(n,\mathbb{H}) \mid
\operatorname{Im}(A) \subset \mathfrak{h} \subset \ker(A) \bigr\}=: \mathfrak{a}.
\]
\item \textbf{Step 4:} Application of the Ambrose-Singer Theorem. \par
We recall the following version of the Ambrose-Singer Theorem for left-invariant connections.
\begin{theorem}\label{thm:AS2}{\textup{\cite{Al}}}
Let $\nabla$ be a left-invariant affine connection on the Lie group $G$, and let $\g$ denote the Lie algebra of $G$. Then the holonomy algebra $\mathfrak{hol}(\nabla)$ based at the identity element $e\in G$ is the smallest subalgebra of $\mathfrak{gl}(\g)$ containing the curvature endomorphisms $R(X,Y)$ for any $X,Y\in\g$, and closed under commutators with the left multiplication operators $\nabla_X:\g\to \g$.
\end{theorem}
Let us consider $\text{span}\{R(X,Y) \mid X,Y \in \g \}$. As already observed in Step $3$,
\[
\text{span}\{R(X,Y) \mid X,Y \in \g \}\subset \a.
\] 
We point out that $\a$ is an abelian subalgebra of $\mathfrak{sl}(n,\H)$.\par
Let $A \in \a$ and let $X \in \g$.
\[
[ A, \nabla_X]\in \a.
\]
In fact, let $Y \in \g$ and $Y' \in \h$:
\[
A \nabla_X Y -\nabla_X AY\subset A(\g_1^\H)-\nabla_X \h =A(\g_1^\H) \subset \h, 
\]
that is, $\mathrm{Im}([ A, \nabla_X])\subset \h$, and 
\[
A \nabla_X Y' -\nabla_X AY'=0, 
\]
implying that $\h \subset \ker([ A, \nabla_X])$. Moreover, $[A,\nabla_X]\in \mathfrak{sl}(n,\H)$ since $\mathfrak{sl}(n,\H)$ is an ideal of $\mathfrak{gl}(n,\H)$.\par
Since $\a$ is closed under commutators with $\nabla_\g$, by Theorem \ref{thm:AS2} we get that $\mathfrak{hol}(\nabla) \subset \a$. In particular, $\a$ is abelian and has trivial product.
\end{itemize}
\end{proof}
\begin{example}
Let $(\mathfrak{g}, \{J_\alpha\})$ be the $2$-step nilpotent hypercomplex Lie algebra constructed in Example~\ref{ex:nonflat} and let us consider the corresponding $2$-step nilpotent hypercomplex Lie group $G$. As already observed, the Obata connection $\nabla$ on $G$ is not flat. A direct computation shows that the holonomy algebra $\mathfrak{hol}(\nabla)$ is given by
\[
\mathfrak{hol}(\nabla) = \left\{ 
\begin{pmatrix}
\begin{array}{c|c|c}
0 & 0 & 0 \\ \hline
0 & 0 & 0 \\ \hline
q & p \cdot j & 0
\end{array}
\end{pmatrix}
\in \mathfrak{sl}(3, \mathbb{H}) \ | \ q \in \mathbb{H},\ p \in \mathbb{R} 
\right\},
\]
where we are considering $\g=\text{span}_\H \{e_1,e_5,e_9\}$.\par
This provides an explicit example of a holonomy algebra arising from a non-flat Obata connection on a $2$-step nilmanifold such that $\mathfrak{hol}(\nabla)$ is abelian and non-trivial.
\end{example}

The next lemma is one of the key ingredients in the proof of Theorem \ref{thm:trivialhol} below.

\begin{lemma} \label{lem:nabla_nil}
Let $G^{4n}$ be a $2$-step nilpotent Lie group endowed with an invariant hypercomplex structure $\hcx$. Then for any $X,Y \in \g$, $\nabla_X Y \in \g_1^\H$ and $\nabla_X \g_1^\H \subset \h$. In particular, for any $X\in \g$, $\nabla_X$ is nilpotent. 
\end{lemma}
\begin{proof}
The first claim is straightforward by Formula \eqref{eqn:Obata}.\par
For the second claim, since $\nabla_X J_\alpha=J_\alpha \nabla_X$ and $\h$ is $\hcx$-invariant, it suffices to prove that $\nabla_X \g^1 \subset \h$. Applying Formula \eqref{eqn:Obata}, we see that for any $Y \in \g^1$:
\[
\nabla_X Y=\frac{1}{2}\big(- J_2[X,J_2Y]+J_3[J_1 X, J_2 Y]\big).
\]
With the same argument used in the proof of Step 1 in Theorem \ref{thm:holonomy}, the first term lies in $\g^1 \cap J_2 \g^1$, while the second term lies in $J_1 \g^1 \cap J_3 \g^1$. The proof follows by the definition of $\h$.\par
The fact that $\nabla_X$ is nilpotent follows since $(\nabla_X)^3=0.$ In fact for any $Y \in \g$
\[
\nabla_X(\nabla_X(\nabla_X Y))\in \nabla_X(\nabla_X \g_1^\H) \subset \nabla_X \h=0,
\]
by Step 2 in the proof of Theorem \ref{thm:holonomy}.
\end{proof}

\subsection{Trivial Obata holonomy} Let $G$ be a $2$-step nilpotent Lie group endowed with an invariant hypercomplex structure $\hcx$. If the Obata connection is flat, then the full holonomy group is discrete. A natural question is when the holonomy group is trivial; in this section we prove in Theorem~\ref{thm:trivialhol} below that it is never the case.

\smallskip

In order to prove this, we recall some useful facts on affine connections. Let $M$ be a connected $n$-dimensional manifold with an affine connection $\nabla$. We define the following subspace of the space $\mathfrak X (M)$ of smooth vector fields on $M$:
\begin{equation}\label{eqn:p-nabla} 
\mathcal P^\nabla=\{X \in \mathfrak{X}(M)\mid \nabla X =0\}.
\end{equation}
The subspace $\mathcal P^\nabla$ is in general neither finite dimensional as an $\R$-vector space nor closed under the Lie bracket of vector fields. However, under additional assumptions both conditions are satisfied, as the following well known lemma shows  (see \cite[Proposition 2.2]{Wolf1} for the first item and \cite[p. 323]{Wolf1} for the second one).
\begin{lemma} [\cite{Wolf1}] \label{lem:parallel-is-Lie-subalg}
Let $M$ be a connected $n$-dimensional manifold with an affine connection $\nabla$.
\begin{enumerate}
    \item[$(1)$] If the holonomy group $\operatorname{Hol}(\nabla)$   is trivial, then $\mathcal P^\nabla$ is an $n$-dimensional real vector space. 
    \item[$(2)$] The space $\mathcal P^\nabla$ is a Lie subalgebra of $\mathfrak{X}(M)$ if and only if the torsion $T$ of $\nabla$ is parallel.
\end{enumerate}
\end{lemma}

In case the manifold admits an almost complex structure $J$ and an affine connection $\nabla$ with trivial holonomy, we will make use of the following result.

\begin{lemma}[{\cite[Lemma 3.2]{ABD}}] \label{lem:paralle-is-J-stable}
    Let $M$ be a connected manifold with an almost complex structure $J$. Assume that there exists an affine connection $\nabla$ on M such that  $\operatorname{Hol}(\nabla)$   is trivial. Then the following conditions are equivalent: 
    \begin{enumerate}
        \item[$(1)$]
    $\nabla J = 0$; 
    \item[$(2)$] the space $\mathcal P^\nabla$  of parallel vector fields on $M$ is $J$-invariant, that is,  $J (\mathcal P^\nabla )=\mathcal P^\nabla $.
    \end{enumerate}
\end{lemma}

Any Lie group $H$ admits a natural left-invariant connection with trivial holonomy: it is defined by $\nabla^0_XY=0$ for all $X,Y$ left-invariant vector fields on $H$, and it is
known as the $(-)$-connection (see, for instance, \cite[Chapter X, Proposition~2.12]{KN}). Note that its torsion is given by $T^0(X,Y)=-[X,Y]$ for $X,Y$ left-invariant vector fields on $H$. Moreover, $\nabla^0$ is flat, complete, and clearly  $\nabla^0T^0=0$ by definition of $\nabla^0$. \par
If $\Lambda \subset H$ is any discrete subgroup of $H$, then $\nabla^0$ on $H$ induces a unique connection, still denoted $\nabla^0$, on $\Lambda \backslash H$ such that the projection $\pi:H\to\Lambda \backslash H$ is an affine local diffeomorphism. This connection $\nabla^0$ on $\Lambda \backslash H$ is geodesically complete, has trivial holonomy and its torsion is parallel. \par
If $J_0$ is a left-invariant complex structure on $H$ then, by definition of $\nabla^0$, it follows that $\nabla^0 J_0=0$. The complex structure $J_0$ on $H$  induces a unique complex structure, still denoted $J_0$, on $\Lambda \backslash H$ such that $\pi$ is a local biholomorphism. Since $\pi$ is also affine, we  have that $\nabla^0 J_0=0$ on $\Lambda \backslash H$.  Therefore, $(\Lambda\backslash H,J_0)$ carries a geodesically  complete connection $\nabla^0$
with trivial holonomy and parallel torsion such that $\nabla^0 J_0=0$. See \cite{ABD} for further details.

\begin{theorem} [{\cite[Theorem 4.1]{ABD}}] \label{discrete}
The triple $(M,J,\nabla)$ where $M$ is a connected manifold endowed with a complex structure $J$ and a connection
$\nabla$ with $\nabla J=0$ and trivial holonomy is affinely biholomorphic to a triple $(\Lambda\backslash H, J_0, \nabla^0)$ as above if and only if $\nabla$ is geodesically complete and its torsion is parallel.\par
The Lie group $H$ is the connected and simply connected Lie group with Lie algebra $\mathcal P^\nabla$ defined in \eqref{eqn:p-nabla}, and $J_0$ is the complex structure on $\Lambda \backslash H$ induced by the restriction of $J$ to $\mathcal P^\nabla$.
\end{theorem}

As a consequence of Lemmas \ref{lem:nabla_nil}, \ref{lem:parallel-is-Lie-subalg} and \ref{lem:paralle-is-J-stable}  and Theorem \ref{discrete}  we obtain the main result of this section.

\begin{theorem}  \label{thm:trivialhol}
Let $(\Gamma \backslash G,\hcx)$ be an at most 2-step nilmanifold endowed with an invariant hypercomplex structure $\hcx$ such that the Obata connection $\nabla$ is flat. Then the holonomy group $\operatorname{Hol}(\nabla)$ is trivial if and only if $\Gamma \backslash G$ is a torus. \par
In particular, if $(\Gamma \backslash G,\hcx)$ is a $2$-step nilmanifold endowed with an invariant hypercomplex structure such that the Obata connection $\nabla$  is flat, then the holonomy group $\operatorname{Hol}(\nabla)$ is never trivial.
\end{theorem}
\begin{proof} 
The implication from right to left is clear. Let us now prove the converse. \par
The Obata connection $\nabla$ satisfies $\nabla J_\alpha =0$ for all $\alpha$ and has vanishing (and hence parallel) torsion. By assumption, $\operatorname{Hol}(\nabla)$ is trivial, therefore, Lemma \ref{lem:parallel-is-Lie-subalg}  implies that the space $\mathcal P^\nabla$ of smooth parallel vector fields on $\Gamma\backslash G$ is a $4n$-dimensional Lie algebra, with $4n:=\dim \Gamma\backslash G =\dim G$, and it follows from Lemma \ref{lem:paralle-is-J-stable} that $J_\al (\mathcal  P^\nabla )=\mathcal P^\nabla $ for all $\al$. We denote by $H$  the connected and simply connected Lie group with Lie algebra $\mathcal P^\nabla$.    

On the other hand, it follows from Lemma \ref{lem:nabla_nil} that, for any $X \in \g$, the endomorphism $\nabla_X: \g \to \g$ is nilpotent. As the Obata connection $\nabla$ is flat and torsion-free, it follows from \cite[Theorem 2.2 and Proposition 1.2]{Kim} that $\nabla$ is geodesically complete on $G$, hence on $\Gamma\backslash G$.  \par
Since $\nabla$ is flat, torsion-free, geodesically complete and $\nabla J_\alpha =0$ for all $\al$, we may apply Theorem \ref{discrete} to conclude that $(\Gamma \backslash G,J_\alpha, \nabla)$ is affinely biholomorphic to $(\Lambda \backslash H,J_\alpha^0, \nabla^0)$ for each $\al$, where  $\Lambda$ is  a lattice in $H$, $J_\alpha^0$ is the complex structure induced by $J_\alpha$ on $\Lambda \backslash H$  and $\nabla^0$ is the $(-)$-connection on $\Lambda \backslash H$. \par
As $\nabla$ is torsion free, so is $\nabla^0$, therefore,  we have that $[X,Y]=0$ for any  $X, Y$ vector fields on $\Lambda \backslash H$ induced by left-invariant vector fields on $H$. In other words, $[X,Y ]=0$ for any $X, Y \in \mathcal P^\nabla$. In particular, $H$ is abelian and therefore $\Lambda \backslash H$ is a torus. \par
Since $\Gamma \backslash G$ is diffeomorphic to  $\Lambda \backslash H$ and both $G$ and $H$ are nilpotent, then $G$ is isomorphic to $H$ as a Lie group \cite{Mal}. It then follows that $G$ is abelian, and so $\Gamma \backslash G$ is also a torus. 
\end{proof}
\begin{remark}
Theorem \ref{thm:trivialhol} still holds true if instead of considering a 2-step nilmanifold we assume that $\Gamma\backslash G$ is a solvmanifold with $G$ a completely solvable almost abelian Lie group. Indeed, in this case the Obata connection is automatically flat, according to \cite[Proposition 3.7]{AB1}, and it  is complete, due to \cite[Corollary 3.8]{AB1}. For the last paragraph, we use the Saito rigidity theorem \cite{Sai}, which extends the result of Mal'cev from nilpotent to completely solvable Lie groups.
\end{remark}

\section{Finale} If $(\g, \hcx )$ is a hypercomplex Lie algebra such that $\g$ is $2$-step nilpotent, we have seen that the Obata connection is flat when at least one complex structure $J_\alpha$ satisfies 
\begin{equation}  \label{eqn:J-gen-2-step}  
J_\alpha \mathfrak{g}^1 \subset \mathfrak{z}.\end{equation} We now investigate to what extent this condition can be extended to Lie algebras of higher nilpotency step. It turns out that \eqref{eqn:J-gen-2-step} forces the Lie algebra to be at most $3$-step nilpotent.

\begin{lemma} \label{lem:at-most-3-step} 
\begin{enumerate} \item[]
    \item[$(1)$]  
Let $(\g, J)$ be a Lie algebra endowed with a complex structure $J$ satisfying condition \eqref{eqn:J-gen-2-step}. Then $\g$ is at most $3$-step nilpotent and $\g^1$ is abelian.
\item[$(2)$] Let $(\g, \{J_1, J_2, J_3\})$ be a hypercomplex Lie algebra  and assume that $(\alpha , \beta , \gamma)$ is a cyclic permutation of $(1,2,3)$ such that   $J_\alpha$ and $J_\beta$    satisfy \eqref{eqn:J-gen-2-step}.   Then  $J_\gamma \g^1\subset \z (\g^1 )$ and  $\nabla_X Y =0$ for any $Y\in \g^1$, where $\nabla$ is the Obata connection.
    \end{enumerate}
\end{lemma}
\begin{proof}
 In case (1), by the integrability of $J$ and the assumption $J \g^1 \subset \z$, one obtains
\begin{equation} \label{eqn:integrability3}
[[X,Y], Z] = -J[[X,Y], J Z], \quad \text{for any } X,Y,Z\in \g,
\end{equation}
which shows that $\g^2 \subset \z$, and hence $\g^3 = [\g^2, \g] = 0$. Note that by Equation \eqref{eqn:integrability3}, $[\g^1,\g^1]=0$. \par 
To prove (2), we may assume that $J_1$ and $J_2$ satisfy \eqref{eqn:J-gen-2-step}. We start by showing that $J_3 \g^1\subset  \z (\g^1 )$, that is,  $[[X,Y],J_3 [ U, Z]]=0$ for any $X,Y,U,Z\in \g $. Indeed, using the integrability of $J_1$:
\begin{align*}
[[X,Y],J_3[U,Z]] & = -J_1[J_1[X,Y],J_3[U,Z]]-J_1[[X,Y],J_1J_3[U,Z]]+[J_1[X,Y],J_1J_3[U,Z]]\\
& = -J_1[J_1[X,Y],J_3[U,Z]]+J_1[[X,Y],J_2[U,Z]]-[J_1[X,Y],J_2[U,Z]]\\
& = 0,
\end{align*}
since $J_1 \g^1 \subset \z$ and $J_2 \g^1 \subset \z$.
The second part of (2) follows by computing   $\nabla _X [V,W]$ for  any $X, V, W \in \g$ using \eqref{eqn:Obata} and by applying  the fact that both $J_1$ and $J_2$ satisfy \eqref{eqn:J-gen-2-step} and \eqref{eqn:integrability3}.
\end{proof}

Since in the $2$-step case the corresponding Obata connection is flat (see Corollary~\ref{cor:flat}), we will focus our attention on the $3$-step case. We observe that the hypercomplex structure must be non-abelian. Indeed, if $\hcx$ is abelian and $J_\alpha \g^1 \subset \z$ for at least one $\alpha$, then
\[
\g^1 \subset J_\alpha \z = \z,
\]
which forces $\g$ to be $2$-step nilpotent.\par 
Examples of 3-step nilpotent Lie algebras with hypercomplex structures have been given in \cite{DF2,BD,AB2}. Among these, the only examples with non-abelian hypercomplex structures are those in \cite{AB2}; however, those are Obata flat and none of them satisfies the hypothesis of Lemma \ref{lem:at-most-3-step} (2). \par
 
We proved in Theorem~\ref{thm:holonomy} that for hypercomplex 2-step nilmanifolds the holonomy of the Obata connection is an abelian Lie algebra, which is also  a commutative associative algebra with trivial product. This theorem requires no additional assumption on the hypercomplex structure. Next, we prove  a similar result  without assuming 2-step nilpotency, but requiring an extra condition on the hypercomplex structure, namely, that $\{J_1, J_2 , J_3\} $ satisfies the hypothesis of Lemma~\ref{lem:at-most-3-step} (2).

\begin{proposition} \label{prop:holonomy2} 
Let $(\g,\{J_1, J_2 , J_3\})$ be a $4n$-dimensional hypercomplex Lie algebra  and assume that $(\al , \beta ,\gamma)$ is a cyclic permutation  of $(1,2,3)$   such that $J_\al$ and $J_\beta$    satisfy \eqref{eqn:J-gen-2-step}. Then:
\begin{enumerate} 
\item[$(1)$] the curvature tensor of the Obata connection is given by 
$R(X,Y)=-\ad_{[X,Y]}$,
hence  $(\g,\hcx)$ is non-flat if and only if $\g$ is $3$-step nilpotent; 
\item[$(2)$]   the holonomy algebra $ \mathfrak{hol}(\nabla)$ of the Obata connection $\nabla$ on $G$ is an abelian subalgebra of $\mathfrak{sl}(n,\mathbb{H})$, where  $G$ is the  connected and simply connected  Lie group with Lie algebra $\g$. Moreover, $A \cdot B=0$ for any $A,B \in \mathfrak{hol}(\nabla)$, therefore, $\mathfrak{hol}(\nabla)$  is  a commutative associative algebra with trivial product.
\end{enumerate} 
\end{proposition}
\begin{proof}
To prove (1), we first show that $\nabla_X \nabla_Y Z =0$. 
We start by computing $\nabla_Y Z$  
applying Formula \eqref{eqn:Obata}. Since $\nabla_X$ commutes with $J_\al$ for all $\alpha$, Lemma~\ref{lem:at-most-3-step} (2) implies that $\nabla_X \nabla_Y Z =0$.
Therefore,
\[
R(X,Y)Z=-\nabla_{[X,Y]}Z=-\nabla_Z [X,Y] +[Z,[X,Y]]= - [[X,Y], Z],
\]
where the second equality follows from the fact that $\nabla$ is torsion-free.  \par
In order to prove (2), set \[ \mathfrak{b}:=\text{span}\{\ad_{[X,Y]} \mid X,Y \in \g\}  .\] Next, we show  that  $\mathfrak{b}$ is a commutative associative  algebra with trivial product and that $\mathfrak{hol}(\nabla)=\mathfrak b $. 
 In fact, 
\[
\ad_{[X,Y]}\circ\ad_{[Z,W]}U=[[X,Y],[[Z,W],U]]=0, 
\]
as $\g^1$ is abelian (Lemma \ref{lem:at-most-3-step} (1)). Therefore, $\mathfrak{b}$ is a commutative associative   algebra with trivial product.

Next, we show  that $ \mathfrak{hol}(\nabla)=\mathfrak b $. 
From Lemma~\ref{lem:at-most-3-step} (2), we have that 
\[ 
\nabla_Z \ad_{[X,Y]} =0. \]
On the other hand,    
\begin{align*}
2\ad_{[X,Y]} \nabla_Z  U&=\ad_{[X,Y]}J_3 [{J_1 Z} , J_2 U] = 0,   
\end{align*}
where the last equality follows from $J_3\g^1\subset \z (\g^1)$ (Lemma \ref{lem:at-most-3-step} (2)). 
By Theorem \ref{thm:AS2}, $\mathfrak{hol}(\nabla)=\mathfrak{b}$, which implies that $\mathfrak{hol}(\nabla)$ is an abelian Lie subalgebra of $\mathfrak{sl}(n,\H)$, which is also an associative commutative algebra with trivial product, concluding the proof.
\end{proof}

With techniques similar to those in Section \ref{sec:examples}, we can construct an example of a $3$-step nilpotent Lie algebra endowed with a hypercomplex structure $\hcx$ such that $J_\alpha \g^1 \subset \z$ for each~$\alpha$.
\begin{example}
Let $\g$ be the $3$-step nilpotent Lie algebra defined by the structure equations
\[
\begin{split}
    &de^i =0, \ i=1,\dots,4, \ \ de^5 =e^{1,2}-e^{3,4}, \\ 
    &de^i =0, \ i=6,7,8, \ \  de^9  =e^{1,3}-e^{4,2}, \\
    &de^{i}=0,\  i=10, 11, 12, \ \ de^{13}=-e^{2,5}+e^{3,9}, \\
    &de^{14}=e^{1,5}+e^{4,9}, \ \ de^{15}=e^{4,5}-e^{1,9}, \ \ de^{16}=-e^{3,5}-e^{2,9}.
\end{split}
\]
Then $\g$ can be endowed with the following hypercomplex structure 
\begin{align*}
&J_1 e_1=e_2, \ J_1 e_3=e_4, \ J_1 e_5= e_6, \ J_1 e_7=e_8, \\
&J_1 e_9=e_{10}, \ J_1 e_{11}=e_{12},\, J_1 e_{13}=e_{14}, \ J_1e_{15}=e_{16}, \nonumber\\
&J_2 e_1=e_3, \ J_2 e_2=-e_4, \ J_2 e_5= e_7, \ J_2 e_6=-e_8, \\ 
&J_2 e_9=e_{11}, \ J_2 e_{10}=-e_{12}, J_2 e_{13}=e_{15}, \ J_2 e_{14}=-e_{16},\\
&J_3e_1=e_4, \ J_3 e_2=e_3, \ J_3 e_5= e_8, \ J_3 e_6=e_7, \\ 
&J_3 e_9=e_{12}, \ J_3 e_{10}=e_{11}, \ J_3 e_{13}=e_{16}, \ J_3 e_{14}=e_{15}, \nonumber
\end{align*}
satisfying $J_\alpha \g^1 \subset \z$, for each $\alpha=1,2,3$. Moreover, since $\g$ is $3$-step, Proposition~\ref{prop:holonomy2}  implies that the Obata connection on the associated connected and simply connected nilpotent Lie group is not flat. 
\end{example}
Theorem \ref{thm:holonomy} and Proposition \ref{prop:holonomy2} highlight an unexpected rigidity in the holonomy algebra of the Obata connection.  A related result is proved in \cite{Ge}, where the author shows that if the hypercomplex structure is abelian then the holonomy algebra of the Obata connection is abelian. \par
In the $3$-step nilpotent setting, as recalled above, all previously known examples were either Obata-flat or equipped with an abelian hypercomplex structure (the example constructed in this section is neither Obata-flat nor endowed with an abelian hypercomplex structure). Nevertheless, in every known case, the holonomy algebra of the Obata connection is abelian. This recurring behavior suggests that such rigidity might be a general feature of the nilpotent setting, motivating the following question:
\begin{question}
Let $G$ be a nilpotent Lie group endowed with a left-invariant hypercomplex structure. Is it true that the holonomy algebra of the Obata connection is always abelian?
\end{question}


\noindent
\begin{flushleft}

\fontsize{10pt}{10pt}\selectfont
A. Andrada, M. L. Barberis:  \textsc{FAMAF, Universidad Nacional de Córdoba and CIEM-CONICET,}\\
\textsc{Av. Medina Allende s/n, Ciudad Universitaria, X5000HUA Córdoba, Argentina}\\
\texttt{adrian.andrada@unc.edu.ar} \\
\texttt{mlbarberis@unc.edu.ar}

\vspace{0.9em}
B. Brienza: \textsc{Dipartimento di Matematica ``G. Peano'', Università degli Studi di Torino,}\\
\textsc{Via Carlo Alberto 10, Torino, Italy}\\
\texttt{beatrice.brienza@unito.it}
\end{flushleft}

\end{document}